\documentclass[11pt]{amsart}

\usepackage{amsmath, amscd, amssymb, amsthm}
\usepackage{latexsym, amssymb, amsmath, amscd, amsfonts}
\usepackage{url}
\usepackage{tikz}
\usepackage[T1]{fontenc}
\usepackage{multirow}
\usepackage{comment}
\usepackage{makecell}

\usepackage{colonequals}
\usepackage{calc}
\usepackage{tikz}
\usetikzlibrary{decorations.markings}
\tikzstyle{vertex}=[circle, draw, inner sep=0pt, minimum size=6pt]
\newcommand{\vertex}{\node[vertex]}

\numberwithin{equation}{section}
\newtheorem{theorem}{Theorem}[section]
\newtheorem{corollary}[theorem]{Corollary}
\newtheorem{definition}[theorem]{Definition}
\newtheorem{example}[theorem]{Example}
\newtheorem{lemma}[theorem]{Lemma}
\newtheorem{proposition}[theorem]{Proposition}
\newtheorem{remark}[theorem]{Remark}
\newtheorem{notation}[theorem]{Notation}

\begin{document}

\title
{Graph Laplacians with Higher Accuracy}

\author{Mary Yoon}
\email{ymary1201@gmail.com}
\address{Department of Mathematics, Korea University, 
145 Anam-ro Seongbuk-gu, Seoul 02841, South Korea}

\begin{abstract}
Motivated by discrete Laplacian differential operators with various accuracy orders in numerical analysis, we introduce new matrices attached to a simple graph that can be considered graph Laplacians with higher accuracy. 
In particular, we show that the number of graphs having cospectral mates with these matrices is significantly less than the ones with other known matrices. 
We also investigate their spectral properties and explicitly compute their eigenvalues and eigenvectors for some graphs. 
Along the line, we also prove the existence of a weighted signed graph with given Laplacian eigenvalues.
\end{abstract}

\subjclass[2020]{05C50, 05C22, 15A18, 15B99}
\keywords{Graph theory, $m$-Laplacian matrix, $m$-Laplacian eigenvalue, Weighted signed graph, Laplacian matrix}

\maketitle


\tableofcontents

\section{Introduction} \label{sec-introduction}

Let $G$ be a simple graph with $n$ vertices $v_1, ..., v_n$. 
Recall that the {adjacency matrix} $A_G$ of $G$ is 
the $n \times n$ matrix whose $(i,j)$ entry is $1$ 
if $v_i$ and $ v_j$ are adjacent; and $0$ otherwise. 
The {degree matrix} $D_G$ of $G$ is the $n \times n$ diagonal matrix 
whose $(i,i)$ entry is the degree of $v_i$, i.e., the number of 
edges incident with $v_i$.  Then, the {Laplacian matrix}, 
or {graph Laplacian} of $G$ is  
\[
L_G = D_G - A_G.
\] 
In this paper, we introduce a more accurate version of the Laplacian matrix of a simple graph $G$, called the \underline{$m$-Laplacian} of $G$ and denoted by $L^{(m)}_G$ 
for $1 \leq m < n$. We then investigate its spectral properties.

Our motivation is from the finite difference method in numerical analysis.  
We will show that when $G$ is a cycle of length $n$, its $m$-Laplacian can be considered an approximation of the one-dimensional Laplacian differential operator with the order of accuracy $2m$, 
while the ordinary Laplacian is an approximation with the order of accuracy $2$, thus $m=1$. 
Indeed, our $1$-Laplacian of any simple graph $G$ is precisely the ordinary Laplacian of $G$
\[
L^{(1)}_G = L_G.
\]

\medskip

Spectral graph theory examines the properties of graphs using  
the eigenvalues of matrices associated with them. 
Thus, knowing whether these eigenvalues can uniquely characterize graphs is important.  
Recently, the spectral properties of so-called signless Laplacians of graphs have received considerable attention, partly due to the relatively small number of graphs sharing the same spectrum (\cite{C Signless}). 
See, for example, \cite{Belardo signless, Chen signless, Cui signless, 
Gho signless, Lin signless, Oboudi signless, Zhao signless}.
Table \ref{tab-cospectral number} shows that the number of graphs with the same eigenvalues of our $2$-Laplacian is significantly lower than that of the signless Laplacian and other matrices.
Even when the number of vertices is $9$ or less, there is no cospectral pair with respect to $3$-Laplacian. 
That means all graphs with $9$ or less vertices is determined by its $3$-Laplacian spectrum. 
Thus, we strongly believe that the $m$-Laplacians can be potentially more powerful tools than the ones currently studied in algebraic graph theory.

\begin{table} 
\begin{center}
\begin{tabular}{|c||c|c|c|c|c|c|}
\hline
\hspace*{1.5mm}$|V|$ \hspace*{1.5mm}& \hspace*{1.5mm} $G$ \hspace*{1.5mm} & \hspace*{1.5mm} $A$ \hspace*{1.5mm}& \hspace*{1.5mm} $L$ \hspace*{1.5mm} & \hspace*{1.5mm} $|L|$\hspace*{1.5mm} & \hspace*{1.45mm} $L^{(2)}$ \hspace*{1.45mm}  & \hspace*{1.45mm} $L^{(3)}$ \hspace*{1.45mm} \\
\hline
1  &  1 &    0  &   0   &  0 &   0   & 0      \\
2  &  2 &    0  &   0   &  0 &   0   & 0      \\
3  &  4 &    0  &   0   &  0  &  0   & 0      \\
4  &  11 &   0   &  0    & 2  &  0   & 0       \\
5  &  34 &   2   &  0    &  4  &   0     & 0   \\
6  &  156 &  10    &   4   &  16  &   0    & 0    \\
7  &  1044 &  110    &   130   & 102 &   2  & 0        \\
8  &  12346 &  1722    &   1767   & 1201  & 2  & 0         \\
9  &  274668 &  51039    &   42595   &  19001 & 4  & 0        \\
\hline
\end{tabular}
\end{center}
\caption{The number of cospectral mates with respect to various graph matrices: 
The number of vertices is listed in the first column $|V|$. 
The numbers of non-isomorphic simple graphs with a given number of vertices are listed in the second column $G$. The number of graphs with cospectral mates with respect to the adjacent, Laplacian, signless Laplacian, $2$-Laplacian and the $3$-Laplacian matrices are listed respectively in columns $A$, $L$, $|L|$, $L^{(2)}$ and $L^{(3)}$.} 
\label{tab-cospectral number}
\end{table}

\medskip

The $m$-Laplacian of a simple graph $G$ can be defined as the Laplacian of an edge-weighted graph $G_m$ derived from $G$. 
Here, edge weights on $G_m$ are not necessarily the same sign, and thus, $G_m$ is a so-called weighted \textit{signed} graph. 
In this paper, we call it simply a \underline{weighted graph}.
Recently, many researchers tend to pay more attention to the Laplacian spectra of weighted graphs and their applications 
(for example,  
\cite{Chen signed, Kun signed, Ni signed, Song signed, Zelazo signed}), and our results can be considered extensions and supplements to the Laplacian spectral theory of weighted graphs. 
In particular, we will prove the existence of a weighted graph with some Laplacian spectrum. 

\medskip

The remainder of this paper is organized as follows.
In Section \ref{sec-discrete laplacian operator}, we define a discrete version of the Laplacian differential operator with high accuracy.
In Section \ref{sec-$m$-Laplacian of simple graphs}, we define the $m$-Laplacian of a simple graph and prove its basic properties.
In Section \ref{sec-$m$-Laplacian spectrum}, we examine the $m$-Laplacian spectrum of cycles, stars, and complete graphs. 
In Section \ref{sec-existence of a signed graph}, we prove the existence of a weighted graph with a given Laplacian spectrum.
In Section \ref{sec-$2$-Laplacian of simple graphs}, we focus on the $2$-Laplacian and investigate its properties by focusing on the similarities and differences between $2$-Laplacian and the graph Laplacian. We also compute the $2$-Laplacian spectrum of circulant graphs.

\bigskip

\section{Discrete Laplacian operator of higher accuracy order} 
\label{sec-discrete laplacian operator}

In this section, we find a discrete version 
of the Laplacian differential operator with arbitrary accuracy order using the finite difference method.

\subsection{Discrete Laplacian operator}

Let us consider a function $u(x)$ on the domain 
$\Omega = \{ x \in \mathbb{R} : 0 < x < L \}$.
We first discretize the domain with
\[
x_i = \left( i - \frac{1}{2}  \right) h  \quad \text{for  $1\leq i \leq n$}
\] 
where $n$ is the number of grid points and $h=L/n$. 
For each integer $1\leq i \leq n$, we let 
\[
u_{i}=u(x_i), \quad  u'_i = u'(x_i), \quad u''_i = u''(x_i)
\]
and so on. Here, the indices $i$ will be regarded as elements 
of $\mathbb{Z}_n$, thus 
\[
u_{n+i} = u_i \quad \text{for all $i \in \mathbb{Z}$}.
\]
Then for any integer $i$ and $k$, by Taylor expansion, 
\begin{equation*}
    u_{i+k}=u_{i}+(kh)u'_{i}+\frac{(kh)^2}{2!}u''_{i}+\frac{(kh)^3}{3!}u^{(3)}_{i}+\cdots
\end{equation*}
hence we have 
\begin{equation}\label{eqn-powersum-three}
u_{i+k} - 2 u_{i} + u_{i-k}=(kh)^2 u''_{i}+ 2\times\left\{ 
\frac{(kh)^4}{4!}u^{(4)}_{i}+\frac{(kh)^6}{6!}u^{(6)}_{i}+\cdots \right\}.
\end{equation}

\medskip

Our next task is to find $a_{k,m}$ satisfying 
\begin{equation} \label{goal eqn}
    \sum_{k=1}^{m}{a_{k,m}(u_{i+k}-2u_{i}+u_{i-k})}=h^2 u''_{i}+O(h^{2m+2}).
\end{equation} 
Expanding the left-hand side of equation \eqref{goal eqn}  
using \eqref{eqn-powersum-three}, we obtain
\begin{align*}
&   \sum_{k=1}^{m}{a_{k,m}(u_{i+k}-2u_{i}+u_{i-k})}\\
& = \sum_{k=1}^{m} a_{k,m} \left\{ (kh)^2 u''_{i}+ 2\times \left( 
\frac{(kh)^4}{4!}u^{(4)}_{i}+\frac{(kh)^6}{6!}u^{(6)}_{i}+\cdots \right)\right\} \\  
   &= \ h^2(a_{1,m}+2^2a_{2,m}+3^2a_{3,m}+\cdots+m^2a_{m,m})u''_{i} \\
    & \qquad +2 \times \frac{h^4}{4!}(a_{1,m}+2^4a_{2,m}+3^4a_{3,m}+\cdots+m^4a_{m,m})u^{(4)}_{i}  
    +\cdots \\
    & \qquad +2 \times \frac{h^{2m}}{(2m)!}(a_{1,m}+2^{2m}a_{2,m}+3^{2m}a_{3,m}+\cdots+m^{2m}a_{m,m})u^{(2m)}_{i}+O(h^{2m+2}).
\end{align*}
Therefore, we can find $a_{k,m}$ satisfying 
equation \eqref{goal eqn} by solving the following system of linear equations:
\begin{equation} \label{Laplacian equ}
\begin{bmatrix}
1 & 2^2 & 3^2 & \cdots  & m^2 \\
1 & 2^4 & 3^4 & \cdots  & m^4 \\
1 & 2^6 & 3^6 & \cdots  & m^6 \\
\vdots & \vdots & \vdots & \ddots & \vdots \\
1 & 2^{2m} & 3^{2m} & \cdots  & m^{2m}
\end{bmatrix}
\begin{bmatrix} a_{1,m} \\  a_{2,m} \\ a_{3,m} \\ \vdots \\ a_{m,m} \end{bmatrix} =
\begin{bmatrix} 1 \\  0 \\ 0 \\ \vdots \\ 0 \end{bmatrix}.
\end{equation}
For the rest of the paper, we let
$a_{k,m}$ denote 
the solution to equation \eqref{Laplacian equ}.
Then we have
\[
u''_{i}=\frac{1}{h^2}\sum_{k=1}^{m}{a_{k,m}(u_{i+k}-2u_{i}+u_{i-k})} + O(h^{2m}).    
\]

\medskip

Based on the above discussion, we now define a discrete version of the one-dimensional Laplacian differential operator of high accuracy order. 

\begin{definition} \label{higher ord Laplacian}
For each $m \in \mathbb{Z}_{>0}$, 
the $(2m)$th accuracy order discrete Laplacian operator 
$\Delta^{(m)}$ is 
\begin{equation}\label{def of operator}
 (\Delta^{(m)} u) (x_i) 
 = \frac{1}{h^2}\sum_{k=1}^{m}{a_{k,m}(u (x_{i+k})-2 u(x_{i}) +u (x_{i-k}))}.
\end{equation}
\end{definition}

Using $\Delta^{(m)}$, we can obtain a more accurate numerical solution to, for example, Poisson's equations and heat equations.

\subsection{Explicit values of $a_{k,m}$} 

The expressions of discrete Laplacian operators with some small accuracy order can be found in any numerical analysis reference. 
However, we cannot find its expression with arbitrary accuracy order, and thus we provide an explicit computation of $a_{k,m}$ for the sake of completeness.

\begin{theorem} \label{laplacin expl}
Given $m \in \mathbb{Z}_{>0}$, for each $1\leq k \leq m$, 
\[ 
a_{k,m}=(-1)^{k+1}\frac{2\binom{2m}{m-k}}{k^2\binom{2m}{m}}. 
\]
\end{theorem}

\begin{proof} 
We only need to verify that the above $a_{k,m}$ satisfy 
the system of linear equations given in \eqref{Laplacian equ}. 
It is straightforward to verify the first equation in the system:
\[
\displaystyle\sum_{k=1}^{m}{k^2a_{k,m}}=\frac{2}{\binom{2m}{m}}\displaystyle\sum_{k=1}^{m}{(-1)^{k+1}\binom{2m}{m-k}}=1.
\]
To verify the rest of them, let us successively differentiate 
\[ 
(x-1)^{2m}=\sum_{k=0}^{2m}{(-1)^{k} x^{k}\binom{2m}{k}},
\]
and then substitute $1$ for $x$ to obtain 
\[ 
\sum_{k=0}^{2m}{(-1)^{k+1} k^{\ell}\binom{2m}{k}}=0 
\]
for all $\ell \in \mathbb{Z}_{>0}$. 
Using this, we obtain 
\[ 
\displaystyle\sum_{k=1}^{m}{k^{2j}a_{k,m}}=\frac{2}{\binom{2m}{m}}\displaystyle\sum_{k=1}^{m}{(-1)^{k+1}k^{2j-2} \binom{2m}{m-k}}=0
\]
for $j \geq 2$.
\end{proof}

\begin{table} 
\begin{center}
\begin{tabular}{|c||c|c|c|c|}
\hline
\diaghead(2,-1){\hskip1cm}{$m$}{$k$} & \hspace*{1.7mm} $1$ \hspace*{1.7mm} & 
\hspace*{1.5mm}$2$ \hspace*{1.5mm} & \hspace*{1.5mm} $3$ \hspace*{1.5mm} & 
\hspace*{1.3mm} $4$ \hspace*{1.3mm} \\
\hline \hline
$1$ & $1$ &  & & \\
\hline 
$2$ & $ $4/3$ $ & $-1/12$ & & \\
\hline
$3$ & $3/2$ & $-3/20$ & $1/90$ & \\
\hline
$4$ & $8/5$& $-1/5$ & $8/315$ & $-1/560$\\
\hline 
\end{tabular}
\end{center}
\caption{$a_{k,m}$ for $m=1,2,3,4$.}
\end{table}

\section{$m$-Laplacian of simple graphs} \label{sec-$m$-Laplacian of simple graphs}

In this section, we generalize graph Laplacian using the high accuracy discrete Laplacian differential operator $\Delta^{(m)}$ given in Definition \ref{higher ord Laplacian}.  
From now on, we will treat the operator $\Delta^{(m)}$ as a matrix one can multiply to a column vector 
\[
\mathbf{u} = \begin{bmatrix} u_1 \\ \vdots \\ u_n \end{bmatrix}
=\begin{bmatrix} u(x_1) \\ \vdots \\ u(x_n) \end{bmatrix}
\]
to denote
\[
\Delta^{(m)} \mathbf{u}
= \begin{bmatrix}
(\Delta^{(m)} u ) (x_1) \\ \vdots \\ (\Delta^{(m)} u ) (x_n) \end{bmatrix}.
\]

\subsection{$m$-Laplacians of cycle graphs} 

Let us first focus on $m=1$. Since $a_{1,1}=1$, Definition \ref{higher ord Laplacian} becomes 
 \[
 \Delta^{(1)} \begin{bmatrix}  \vdots \\  u_i  \\ \vdots \end{bmatrix}
 = \frac{1}{h^2} 
 \begin{bmatrix}  \vdots \\ (u_{i+1} + u_{i-1} - 2 u_i)  \\ \vdots \end{bmatrix}.
 \]

On the other hand, note that vertex $v_i$ in the cycle graph $C_n$ is adjacent to $v_{i-1}$ and $v_{i+1}$ for all $i$. 
Therefore, the $i$th row of its adjacency matrix $A_{C_n}$ 
has $1$ at $(i, i-1)$ and $(i, i+1)$; and $0$ otherwise, thus the $i$th row of column vector $A_{C_n} \mathbf{u}$ is
\[
u_{i-1} + u_{i+1}.
\]
Moreover, since $\deg(v_i)=2$ for all $i$,  
the $i$th row of column vector $(D_{C_n} - A_{C_n}) \mathbf{u}$ is
\[
2 u_i - u_{i-1} - u_{i+1}.
\]
This shows that the ordinary graph Laplacian $L_{C_n} = D_{C_n} - A_{C_n}$ 
of $C_n$ gives
\[
L_{C_n} \begin{bmatrix}  \vdots \\  u_i  \\ \vdots \end{bmatrix}
=\begin{bmatrix}  \vdots \\ 2 u_i - u_{i+1}  - u_{i-1}  \\ \vdots \end{bmatrix}
\]
and therefore, we have
\[
L_{C_n} = - h^2  \Delta^{(1)}.
\]

\medskip

Based on  this observation, we want to define the operator $- h^2 \Delta^{(m)}$ as 
the graph Laplacian of the cycle graph $C_n$ with the $(2m)$th order accuracy. 
In order to handle the terms 
\[
u_{i+k} + u_{i-k} 
\]
in equation \eqref{def of operator} using the adjacency of vertices, we introduce the following notation.
\begin{notation} \label{def of A_k}
Let  $A_{k}$ be 
the adjacency matrix of a simple graph with $n$ vertices 
such that $v_i$ and $v_j$ are adjacent if and only if 
\[
j \equiv i \pm k \quad \text{ (mod $n$)}.
\]
\end{notation}

Then the expression in Definition \ref{higher ord Laplacian} can be rewritten in matrix form as follows:
\begin{align} \label{cycle Lap eqn}
\Delta^{(m)}\mathbf{u}=\frac{1}{h^2} \sum_{k=1}^{m}a_{k,m}(A_{k} - 2 I_n) \mathbf{u},
\end{align}
and then using $- h^2 \Delta^{(m)}$ we have
\begin{definition} \label{cycle m-lap}
For each $1 \leq m < n$,
the \underline{$m$-Laplacian} of the cycle graph $C_n$ is
\[
\begin{split}
L^{(m)}_{C_n} \; \colonequals \; & 2(a_{1,m} + a_{2,m} + \cdots + a_{m,m}) I_n \\
\qquad & - (a_{1,m} A_1 + a_{2,m} A_2 + \cdots + a_{m,m} A_m).
\end{split}
\]
\end{definition}

\medskip


\subsection{$m$-Laplacians of simple graphs} 

In order to extend the above definition of the $m$-Laplacian to general simple graphs $G$,
let us generalize $A_{k}$ in Notation \ref{def of A_k} first.

\begin{definition}
The \underline{length $k$ open path matrix} of $G$ for each $1\leq k < n$, denoted by $P_{G,k}$, 
is the $n \times n$ matrix whose $(i,j)$ entry is equal to the number of open paths of length $k$ from $v_i$ to $v_j$ in $G$.  
\end{definition}

Note that the main diagonal elements of $P_{G,k}$ are all zero. 
Here, we emphasize that a path in $G$, represented as a sequence of vertices, 
contains all distinct vertices except possibly the initial and final ones. 
An open path in $G$ is a path in $G$ that starts and ends on different vertices.
We observe that for the cycle graph $C_n$, 
\[
A_{k}=P_{C_{n},k}
\]
for each $1\leq k < n$.

Also note that the definitions of the adjancency and  Laplacian matrices of a simple graph given in Section \ref{sec-introduction} can be easily extended to any edge-weighted graphs in a straightforward way: The {adjacency matrix} $A_{G'}$ of a edge-weighted graph $G'$ with $n$ vertices is the $n \times n$ matrix whose $(i,j)$ entry is the weight on edge $v_i v_j$ and the {degree matrix} $D_{G'}$ of $G'$ is
the $n \times n$ diagonal matrix whose $(i,i)$ entry is the sum of weights on edges $v_i v_k$ over all $v_k$ incident with $v_i$. Then the {Laplacian matrix} of $G'$ can be defined as  
$L_{G'} = D_{G'} - A_{G'}$.

\medskip

\begin{definition} \label{m-lap}
Let $G$ be a {simple graph} with $n$ vertices. 
For each $1\leq m < n$, let $G_{m}$ be the weighted simple graph 
whose adjacency matrix is 
 \[
a_{1,m} P_{G,1} + a_{2,m} P_{G,2} + \cdots + a_{m,m} P_{G,m}
\]
where the coefficients $a_{k,m}$ are the solution to equation \eqref{Laplacian equ}.
Then the 
\underline{$m$-Laplacian}  of $G$, denoted by $L^{(m)}_G$, is the Laplacian matrix of $G_m$ 
\[
L^{(m)}_G \colonequals L_{G_m}.
\]       
\end{definition}

We observe that all the diagonal entries of adjacency matrix of $G_m$ is zero, and the adjacency matrix of $G_m$ can be considered the \underline{$m$-adjacency matrix} of $G$.

We also remark that when $m=1$, $L^{(1)}_G$ is the ordinary graph Laplacian of $G$.  
Also the coefficients $a_{k,m}$ are negative when $k$ is even (Theorem \ref{laplacin expl}). 
Thus, the edge weights on $G_m$ are not necessarily the same sign and therefore 
$L_G^{(m)}$ is the graph Laplacian of a weighted (signed) graph.

\medskip

Note that the graph Laplacian $L_G$ can be appreciated as a matrix version of the negative discrete Laplacian operator acting on graph $G$. 
In this context, we assume that a particular grid point only affects adjacent grid points. 

However, it seems natural to assume that any two connected vertices affect each other, even if they are not adjacent.
For this reason, we consider $A_{k}$ with $k \geq 2$ in Definition \ref{cycle m-lap}. 
The matrix $A_{k}$ describes these interactions.
Additionally, we can reasonably assume that the greater the distance between two vertices, the less impact they have.
As expected, the absolute value of $a_{k,m}$ decreases as $k$ increases.

\begin{example} 
By Theorem \ref{laplacin expl}, $a_{1,3}=\frac{3}{2}$, $a_{2,3}=-\frac{3}{20}$, and $a_{3,3}=\frac{1}{90}$. 
The weighted graph $G_m$ with $G=C_7$ and $m=3$ is given in Figure \ref{fig-C-7-3}.
\end{example}     

\begin{figure}    
 \[
 \begin{tikzpicture}[x=3.2cm, y=3.2cm]
	\vertex (v1) at (51:1) [label=above:$v_{1}$]{};
    \vertex (v2) at (103:1) [label=above:$v_{2}$]{};
	\vertex (v3) at (154:1) [label=left:$v_{3}$]{};
	\vertex (v4) at (206:1) [label=left:$v_{4}$]{};
	\vertex (v5) at (257:1) [label=below:$v_{5}$]{};
	\vertex (v6) at (309:1) [label=below:$v_{6}$]{};
	\vertex (v7) at (360:1) [label=right:$v_{7}$]{};
	\path 
		(v1) edge node{$\frac{3}{2}$} (v2)
        (v2) edge node{$\frac{3}{2}$} (v3)
        (v3) edge node{$\frac{3}{2}$} (v4)
        (v4) edge node{$\frac{3}{2}$} (v5)
        (v5) edge node{$\frac{3}{2}$} (v6)
        (v6) edge node{$\frac{3}{2}$} (v7)
        (v7) edge node{$\frac{3}{2}$} (v1)
        (v1) edge node{$-\frac{3}{20}$} (v3)
        (v2) edge node{$-\frac{3}{20}$} (v4)
        (v3) edge node{$-\frac{3}{20}$} (v5)
        (v4) edge node{$-\frac{3}{20}$} (v6)
        (v5) edge node{$-\frac{3}{20}$} (v7)
        (v6) edge node{$-\frac{3}{20}$} (v1)
        (v7) edge node{$-\frac{3}{20}$} (v2)
        (v1) edge node{$\frac{1}{90}$} (v4)
        (v2) edge node{$\frac{1}{90}$} (v5)
        (v3) edge node{$\frac{1}{90}$} (v6)
        (v4) edge node{$\frac{1}{90}$} (v7)
        (v5) edge node{$\frac{1}{90}$} (v1)
        (v6) edge node{$\frac{1}{90}$} (v2)
        (v7) edge node{$\frac{1}{90}$} (v3)
	;
\end{tikzpicture}
\]
\caption{The graph $G_m$ with $G=C_7$ and $m=3$}
\label{fig-C-7-3}
\end{figure}
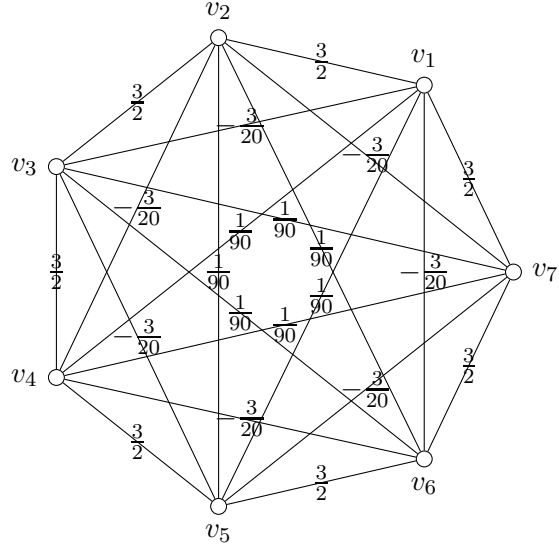

\medskip

\subsection{Properties of $m$-Laplacian} 
From the following proposition, we can see that the $m$-Laplacian $L_{G}^{(m)}$ also has similar properties to the graph Laplacian $L_{G}$.
The following is evident by Definition \ref{m-lap}.

\begin{proposition}
Given two graphs $G$ and $H$ having disjoint vertex sets, let $G \cup H$ be the union of two graphs.
For $1 \leq i \leq n_{1}$ and $1 \leq j \leq n_{2}$, 
let $\lambda_{i}$ and $\mu_{j}$ be the eigenvalues of $L^{(m)}_G$ and $L^{(m)}_G$, respectively. 
Then eigenvalues of $L^{(m)}_{G \cup H}$ are 
\[\lambda_{i} \text{  and  }  \mu_{j}\,\]  
for all $1\leq i \leq n_{1}$ and $1\leq j \leq n_{2}$.
\end{proposition}

\medskip

For $G=C_{n}$, the length $k$ open path matrix $P_{G,k}$ can be written as a power sum of adjacency matrix $A_{C_{n}}$.
This means that $L^{(m)}_{C_{n}}$, like $L_{C_n}$, can be expressed using only the adjacency and diagonal matrices of $C_n$.
We give the following proposition and example without providing proof.

 \begin{proposition} \label{power sum cycle high Lap}
     \begin{equation*}
         L^{(m)}_{C_{n}}= \sum_{k=0}^{m}c_{k,m}A_{C_{n}}^{k},
     \end{equation*}
     where 
     \begin{equation*}
         c_{k,m}= \begin{cases}
2\displaystyle\sum_{j=1}^{[\frac{m+1}{2}]}a_{2j-1,m} +4\displaystyle\sum_{j=1}^{[\frac{m+2}{4}]}a_{4j-2,m}  & \text{ if } k=0 \\ 
\displaystyle\sum_{j=1}^{[\frac{m-k+2}{2}]}(-1)^{j}d_{j,k}a_{2j+k-2,m} & \text{ if } k \geq 1.
\end{cases}
     \end{equation*}
     Here $d_{j,1}=2j-1$, $d_{1,k}=1$, and $d_{j+1,k+1}-d_{j,k+1}=d_{j+1,k}$ for all $j,k \in \mathbb{Z}_{>0}.$
 \end{proposition}

\begin{example} Let $A$ be the adjacency matrix $A_{C_{n}}$ of the cycle graph $C_{n}$. Then
\begin{align*}
        L^{(1)}_{C_{n}}&= 2a_{1,1}I -a_{1,1}A\\
        &=2I-A. \\
        L^{(2)}_{C_{n}}&= 2(a_{1,2}+2a_{2,2})I -a_{1,2}A-a_{2,2}A^2\\
        &=\frac{7}{3}I-\frac{4}{3}A+\frac{1}{12}A^2.\\
        L^{(3)}_{C_{n}}&= 2(a_{1,3}+2a_{2,3}+a_{3,3})I 
         -(a_{1,3}-3a_{3,3})A-a_{2,3}A^2-a_{3,3}A^3\\
        &=\frac{109}{45}I-\frac{22}{15}A+\frac{3}{20}A^2-\frac{1}{90}A^3.\\
        L^{(4)}_{C_{n}}&= 2(a_{1,4}+2a_{2,4}+a_{3,4})I \\
        & \qquad -(a_{1,4}-3a_{3,4})A-(a_{2,4}-4a_{4,4})A^2-a_{3,4}A^3-a_{4,4}A^4\\
        &=\frac{772}{315}I-\frac{32}{21}A+\frac{27}{140}A^2-\frac{8}{315}A^3+\frac{1}{560}A^4.
    \end{align*}
\end{example}

\bigskip

\section{$m$-Laplacian spectrum} \label{sec-$m$-Laplacian spectrum}

In this section, we give explicit expressions of $P_{G,k}$ and then compute eigenvalues and eigenvectors of $L^{(m)}_{G}$ for the cycle, star, and complete graphs.

Before doing that, we first recall the following well-known lemma. 

\begin{lemma} \label{circulant spectrum}
     Consider the $n\times n$ circulant matrix $C$:
     \[
C = \begin{bmatrix}
c_{0} & c_{n-1} & \cdots & c_{1} \\
c_{1} & c_{0} & \cdots & c_{2} \\
\vdots & \vdots & \ddots & \vdots \\
c_{n-1} & c_{n-2} & \cdots & c_{0}
\end{bmatrix}.
\]
The eigenvectors of the matrix $C$ are 
    \[ \mathbf{v}_{j}=(1, \omega^j, \omega^{2j},...,\omega^{(n-1)j})^T, \quad j=0,1,...,n-1,\]
    where $\omega= \exp(\frac{2\pi i}{n})$ and the corresponding eigenvalues are given by
    \[ \lambda_{j}= c_{0}+c_{1}\omega^{j}+c_{2}\omega^{2j}+...+c_{n-1}\omega^{(n-1)j}, \quad j=0,1,...,n-1.\]
\end{lemma}

\medskip

Now we compute the eigenvalues of the $m$-Laplacian of the cycle graph $C_n$. 
 
\begin{proposition} \label{eigenvalue of LapCircle}
    The eigenvectors of $L^{(m)}_{C_{n}}$ are 
    \[ \mathbf{v}_{j}=(1, \omega^j, \omega^{2j},...,\omega^{(n-1)j})^T, \quad j=0,1,...,n-1,\]
    where $\omega= \exp(\frac{2\pi i}{n})$.
    The corresponding eigenvalues are 
    \[ \lambda_{j}= 4\sum_{k=1}^{m}a_{k,m}\sin^2{\frac{\pi kj}{n}}, \quad j=0,1,...,n-1.\]
\end{proposition}
 \begin{proof}
     Consider the $n\times n$ circulant matrix $C$:
     \[
C = \begin{bmatrix}
c_{0} & c_{n-1} & \cdots & c_{1} \\
c_{1} & c_{0} & \cdots & c_{2} \\
\vdots & \vdots & \ddots & \vdots \\
c_{n-1} & c_{n-2} & \cdots & c_{0}
\end{bmatrix}.
\]
    Since $A_{k}$ is a circulant matrix with $c_{k}=c_{n-k}=1$ and others are all zero, by Lemma \ref{circulant spectrum}, the corresponding eigenvalues are 
    \[ \bar{\lambda}_{j}=\omega^{kj}+\omega^{(n-k)j}=2\cos{\frac{2\pi kj}{n}},\]
    for each $j=0,1,...,n-1.$
    Therefore the eigenvectors of $L^{(m)}_{C_{n}}=\sum_{k=1}^{m}a_{k,m}(2I-A_k)$ are $\mathbf{v}_{j}$'s and the corresponding eigenvalues are 
    \[ \lambda_{j}= 4\sum_{k=1}^{m}a_{k,m}\sin^2{\frac{\pi kj}{n}}, \quad j=0,1,...,n-1.\]
 \end{proof}

\medskip

Next, let us compute the eigenvalues and eigenvectors of the $m$-Laplacian of the complete  graph $K_n$. 

 \begin{proposition} \label{spectrum of complete graph}
The eigenvectors of $L^{(m)}_{K_{n}}$ are 
\[ 
\mathbf{v}_{j}=(1, \omega^j, \omega^{2j},...,\omega^{(n-1)j})^T, \quad j=0,1,...,n-1,
\]
where $\omega=\exp(\frac{2\pi i}{n})$.
The corresponding eigenvalues are given by
\[ 
\lambda_{j}= 
\begin{cases}
    0 &\text{  if  } j=0 \\
    n \times \sum_{k=1}^{m}a_{k,m}(k-1)!\binom{n-2}{k-1} &\text{  otherwise.}
\end{cases}
\]     
 \end{proposition}
 \begin{proof}
     Note that 
     \[P_{n,k}=(k-1)!\binom{n-2}{k-1}(J-I), \]
     where $J$ means a $n \times n$ matrix in which all elements are $1$.

By Lemma \ref{circulant spectrum}, the eigenvectors of the matrix $P_{K_n,k}$ are 
    \[ \mathbf{v}_{j}=(1, \omega^j, \omega^{2j},...,\omega^{(n-1)j})^T, \quad j=0,1,...,n-1,\]
    where $\omega=\exp(\frac{2\pi i}{n})$ and the corresponding eigenvalues are given by
    \[ \bar{\lambda}_{j}= 
    \begin{cases}
              (n-1) (k-1)!  \binom{n-2}{k-1} &\text{  if  } j=0 \\
              -(k-1)!  \binom{n-2}{k-1} &\text{  otherwise.}
            \end{cases}\]
     Since
     \[L^{(m)}_{K_{n}}=\sum_{k=1}^{m}a_{k,m}\left\{k!  \binom{n-1}{k}I-P_{K_n,k}\right\},\]
the eigenvectors of $L^{(m)}_{K_{n}}$ are $\mathbf{v}_{j}$ and the corresponding eigenvalues $\lambda_{j}$ are 
\[
    \lambda_{0}=\sum_{k=1}^{m}a_{k,m}\left\{k!\binom{n-1}{k}-(n-1)(k-1)!\binom{n-2}{k-1}\right\} 
    =0
\]
and 
\begin{align*}
    \lambda_{j}&=\sum_{k=1}^{m}a_{k,m}\left\{k!\binom{n-1}{k}+(k-1)!\binom{n-2}{k-1}\right\} \\
    &=n \times \sum_{k=1}^{m}a_{k,m}(k-1)!\binom{n-2}{k-1} 
    \qquad \text{for $j=1, ..., n-1$.}
\end{align*}     
 \end{proof}

\medskip

Our last task is to compute the eigenvalues and eigenvectors of the $m$-Laplacian of the star graph $S_{n-1}$ with $n$ vertices. 
\begin{proposition} 
    The eigenvectors $\mathbf{v}_{j}$ of $L^{(m)}_{S_{n-1}}$ are 
    \[ \mathbf{v}_{j}= 
    \begin{cases}
              (1,1,...,1)^T &\text{  if  } j=1 \\
              (1-n,1,...,1)^T &\text{  if  } j=2 \\
              \epsilon_{j-1}-\epsilon_{j} &\text{  otherwise}
    \end{cases}\]
where $\epsilon_{k}$ is a $n\times 1$ matrix whose $(k,1)$ entry is $1$, and $0$ otherwise.
The corresponding eigenvalues are given by
    \[ \lambda_{j}= 
    \begin{cases}
              0 &\text{  if  } j=1 \\
              na_{1,m} &\text{  if  } j=2 \\
              a_{1,m}+(n-1)a_{2,m} &\text{  otherwise.}
    \end{cases}\]
\end{proposition}
\begin{proof}
Let $v_{1}$ be the center vertex adjacent to all the other vertices and $\mathbf{v}=(c_{1},c_{2},...,c_{n})^T$ be a eigenvector corresponding to eigenvalue $\lambda$ of $L^{(m)}_{S_{n-1}}$.
Then 
\begin{equation} \label{star graph eqn}
\begin{bmatrix}
(n-1)a_{1,m} & -a_{1,m} & \cdots  & -a_{1,m} \\
-a_{1,m} & a_{1,m}+(n-2)a_{2,m} & \cdots  & -a_{2,m} \\
\vdots & \vdots & \ddots & \vdots \\
-a_{1,m} &  -a_{2,m} & \cdots  & a_{1,m}+(n-2)a_{2,m}
\end{bmatrix}
\begin{bmatrix} c_{1} \\  c_{2} \\ \vdots \\ c_{n} \end{bmatrix} = \lambda
\begin{bmatrix} c_{1} \\  c_{2} \\ \vdots \\ c_{n} \end{bmatrix}.
\end{equation}
By expanding \eqref{star graph eqn}, we obtain the following equations.
\begin{align}
    (n-1)a_{1,m}c_{1}-a_{1,m}(c_{2}+c_{3}+...+c_{n})&=\lambda c_{1}, \label{eq1}\\
    -a_{1,m}c_{1}-a_{2,m}(c_{2}+c_{3}+...+c_{n})&=\{\lambda-a_{1,m}-(n-1)a_{2,m} \}c_{k} \label{eq2}
\end{align}
for all $2 \leq k \leq n$.
By equation \eqref{eq2}, 
\[c_{2}=c_{3}=\cdots=c_{n} \text{   or   } \lambda=a_{1,m}+(n-1)a_{2,m}.\]

When $c_{2}=c_{3}=\cdots=c_{n}$, by equation \eqref{eq1}, 
\[c_{2}=\left(1-\frac{\lambda}{(n-1)a_{1,m}}\right)c_{1}.\]
Also, by equation \eqref{eq2},   
\[a_{1,m}c_{1}=\left(a_{1,m}-\lambda\right)\left(1-\frac{\lambda}{(n-1)a_{1,m}}\right)c_{1}.\]
If $c_{1}=0$ then $c_{1}=c_{2}=\cdots=c_{n}=0$, thus 
\begin{align*}
    a_{1,m}&=(a_{1,m}-\lambda)\left(1-\frac{\lambda}{(n-1)a_{1,m}}\right)\\
    &=a_{1,m}-\frac{\lambda}{n-1}-\lambda+\frac{\lambda^2}{(n-1)a_{1,m}}.
\end{align*}
Therefore, $\lambda=0$ or $\lambda=na_{1,m}$ and their corresponding eigenvectors are $(1,1,...,1)^T$ and $(1-n,1,...,1)^T$, respectively. 

\medskip

When $\lambda=a_{1,m}+(n-1)a_{2,m}$, by equation \eqref{eq2}, 
\[a_{1,m}c_{1}+a_{2,m}(c_{2}+c_{3}+...+c_{n})=0.\]
And by equation \eqref{eq1},
\begin{align*}
    (n-1)a_{1,m}c_1 + \frac{a_{1,m}^2}{a_{2,m}}c_{1} &= \lambda c_{1}\\
    &= \{a_{1,m}+(n-1)a_{2,m}\}c_{1}.
\end{align*}
Since $(n-2)a_{1,m}+ \frac{a_{1,m}^2}{a_{2,m}}- (n-1)a_{2,m} \neq 0$, $c_{1}=0$ and $c_{2}+c_{3}+\cdots+c_{n}=0$.
Thus $\epsilon_{j-1}-\epsilon_{j}$ is a eigenvector of $\lambda=a_{1,m}+(n-1)a_{2,m}$ for each $3 \leq j \leq n$.
\end{proof}

 \bigskip

\section{Existence of a weighted graph with a given Laplacian spectrum} \label{sec-existence of a signed graph}

In the proof of Proposition \ref{eigenvalue of LapCircle}, we observe that the eigenvectors of $A_k$'s are the same.
Thus, we can combine them linearly to form a weighted graph with the desired Laplacian spectrum.

 \begin{theorem} \label{existence of signed graphs}
 \hfill
     \begin{enumerate}
         \item Let $\Lambda_{1}$ be a multiset $\{\lambda_{1},\lambda_{1}, \lambda_{2}, \lambda_{2}, ..., \lambda_{m}, \lambda_{m}, 0\}$ with $\lambda_{1}\leq \lambda_{2} \leq \cdots \leq \lambda_{m}$.
         Then there exists a weighted graph $W_1$ with $2m+1$ vertices such that 
         \[spec(L_{W_1})=\Lambda_1.\]
         
         \item Let $\Lambda_{2}$ be a multiset $\{\lambda_{1},\lambda_{1}, \lambda_{2}, \lambda_{2}, ..., \lambda_{m-1}, \lambda_{m-1},0,\lambda\}$ with $\lambda_{1}\leq \lambda_{2} \leq \cdots \leq \lambda_{m-1}$.
         Then there exists a weighted graph $W_2$ with $2m$ vertices such that 
         \[spec(L_{W_2})=\Lambda_2.\]
     \end{enumerate}
 \end{theorem}
\begin{proof} 
    Given $n$, consider a weighted graph $W$ whose adjacency matrix $A_W$ is $c_{1}A_{1}+c_{2}A_{2}+\cdots+c_{m}A_{m}$.
    Then the graph Laplacian of $W$, $L_W$, is
    \[\sum_{k=1}^{m}c_{k}(2I_{n}-A_{k}).\]
    Since the eigenvectors of $A_{k}$'s are the same, as the proof of Proposition \ref{eigenvalue of LapCircle} shows, eigenvalues of $L_W$ are
    \[\bar{\lambda}_{j}=4\sum_{k=1}^{m}c_{k}\sin^2{\frac{\pi kj}{n}} \quad \text{ for } j=0,1,...,n-1.\]
    Note that $\bar{\lambda}_{0}=0$ and $\bar{\lambda}_{j}=\bar{\lambda}_{n-j}$ for all $j=1,2,...,n-1$. 

(1) Suppose that $n=2m+1$.  
        We claim that the following matrix is invertible.
\[
Z = 
\begin{bmatrix}
\sin^2{\frac{\pi}{2m+1}} & \sin^2{\frac{2\pi}{2m+1}} & \cdots & \sin^2{\frac{m\pi}{2m+1}} \\
\sin^2{\frac{2\pi}{2m+1}} & \sin^2{\frac{4\pi}{2m+1}} & \cdots & \sin^2{\frac{2m\pi}{2m+1}} \\
\vdots & \vdots & \ddots & \vdots \\
\sin^2{\frac{m\pi}{2m+1}} & \sin^2{\frac{2m\pi}{2m+1}} & \cdots & \sin^2{\frac{m^2\pi}{2m+1}} 
\end{bmatrix}.
\]
If it is true, we can see that the statement (1) of this theorem is also true because the following matrix equation always has a solution.
\[-4\begin{bmatrix}
\sin^2{\frac{\pi}{2m+1}} & \sin^2{\frac{2\pi}{2m+1}} & \cdots & \sin^2{\frac{m\pi}{2m+1}} \\
\sin^2{\frac{2\pi}{2m+1}} & \sin^2{\frac{4\pi}{2m+1}} & \cdots & \sin^2{\frac{2m\pi}{2m+1}} \\
\vdots & \vdots & \ddots & \vdots \\
\sin^2{\frac{m\pi}{2m+1}} & \sin^2{\frac{2m\pi}{2m+1}} & \cdots & \sin^2{\frac{m^2\pi}{2m+1}} 
\end{bmatrix}
\begin{bmatrix}
    c_{1} \\ c_{2} \\ \vdots \\ c_{m}
\end{bmatrix}
=\begin{bmatrix}
    \lambda_{1} \\ \lambda_{2} \\ \vdots \\ \lambda_{m} 
\end{bmatrix}. 
\] 

Now let us prove the claim. For each $k=1,2,...,m$, let 
$\mathbf{v}_{k}$ be the $k$th column of the matrix $Z$.
We need to show that $\mathbf{v}_{k}$'s are linearly independent. Since 
\[ \mathbf{v}_{2}=4\mathbf{v}_{1}-\begin{bmatrix}
    \sin^{4}{\frac{\pi}{2m+1}} \\ 
    \sin^{4}{\frac{2\pi}{2m+1}} \\ 
    \vdots \\ 
    \sin^{4}{\frac{m\pi}{2m+1}}
\end{bmatrix}, \]
$\mathbf{v}_{2}$ can be replaced by $\mathbf{v}'_{2}=(\sin^4{\frac{\pi}{2m+1}},\;\sin^4{\frac{2\pi}{2m+1}},...,\;\sin^4{\frac{m\pi}{2m+1}})^T$. \\
Similarly, since 
\[ 
\mathbf{v}_{3}=9\mathbf{v}_{1}-24\mathbf{v}'_{2}+16\begin{bmatrix}
    \sin^{6}{\frac{\pi}{2m+1}} \\ 
    \sin^{6}{\frac{2\pi}{2m+1}} \\ 
    \vdots \\ 
    \sin^{6}{\frac{m\pi}{2m+1}}
\end{bmatrix}, 
\]
the vector $\mathbf{v}_{3}$ can be replaced by $\mathbf{v}'_{3}=(\sin^6{\frac{\pi}{2m+1}},\;\sin^6{\frac{2\pi}{2m+1}},...,\;\sin^6{\frac{m\pi}{2m+1}})^T$.\\
In this way, we obtain 
\[
Span\{\mathbf{v}_{1},\mathbf{v}_{2},...,\mathbf{v}_{m}\} = Span\{\mathbf{v}_{1},\mathbf{v}'_{2},...,\mathbf{v}'_{m}\}, \]
 where 
\[\mathbf{v}'_{k}=(\sin^{2k}{\frac{\pi}{2m+1}},\;\sin^{2k}{\frac{2\pi}{2m+1}},...,\;\sin^{2k}{\frac{m\pi}{2m+1}})^T.\]
Suppose that $\sum_{k=1}^{m}d_{k}\mathbf{v}_{k}=0$ for some scalar $d_{k}$'s.
Then 
\begin{align*}
    d_{1}\sin^{2}{\frac{\pi}{2m+1}}+d_{2}\sin^{2}{\frac{2\pi}{2m+1}}+\cdots+d_{m}\sin^{2}{\frac{m\pi}{2m+1}}&=0. \\
    d_{1}\sin^{4}{\frac{\pi}{2m+1}}+d_{2}\sin^{4}{\frac{2\pi}{2m+1}}+\cdots+d_{m}\sin^{4}{\frac{m\pi}{2m+1}}&=0. \\ 
    \vdots \qquad\qquad\qquad\qquad\qquad\qquad\\ 
     d_{1}\sin^{2m}{\frac{\pi}{2m+1}}+d_{2}\sin^{2m}{\frac{2\pi}{2m+1}}+\cdots+d_{m}\sin^{2m}{\frac{m\pi}{2m+1}}&=0. 
\end{align*}
This means $d_{1}x+d_{2}x^{2}+\cdots+d_{m}x^{m}$ has $m+1$ different real roots: 
\[
0, \;\sin^{2}{\frac{\pi}{2m+1}}, \;\sin^{2}{\frac{2\pi}{2m+1}},...,\;\sin^{2}{\frac{m\pi}{2m+1}}.
\]
Thus $d_{1}=d_{2}=\cdots=d_{m}=0$ as we desired.

(2) Now let $n=2m$. The proof of this case can be 
        obtained easily from the invertiblity of the following matrix, which can be shown using similar arguments given in the previous case. 
\[
Z' = \begin{bmatrix}
\sin^2{\frac{\pi}{2m}} & \sin^2{\frac{2\pi}{2m}} & \cdots & \sin^2{\frac{m\pi}{2m}} \\
\sin^2{\frac{2\pi}{2m}} & \sin^2{\frac{4\pi}{2m}} & \cdots & \sin^2{\frac{2m\pi}{2m}} \\
\vdots & \vdots & \ddots & \vdots \\
\sin^2{\frac{m\pi}{2m}} & \sin^2{\frac{2m\pi}{2m}} & \cdots & \sin^2{\frac{m^2\pi}{2m}} 
\end{bmatrix}.
\]
\end{proof}

By applying the Cauchy interlacing theorem, we obtain the following result.

\begin{corollary}
Given $\Lambda=\{\lambda_1, \lambda_2,...,\lambda_m\}$, there is a weighted graph $W$ with $2m$ vertices such that $\Lambda \subset spec(L_W)$.
\end{corollary}
\begin{proof}
    Let $W$ be a weighted graph obtained by deleting a vertex $v_{2m+1}$ from $W_1$ in Theorem \ref{existence of signed graphs}. 
    Let $\mu_{1} \leq \mu_{2} \leq ... \leq \mu_{2m}$ be eigenvalues of $L_{W}$.
    Now suppose that $\lambda_{i}\leq 0 \leq \lambda_{i+1}$.
    Then, by the Cauchy interlacing theorem, 
    \[ \lambda_{j}= 
    \begin{cases}
              \mu_{2j-1} &\text{  for  } 1\leq j \leq i \\
              \mu_{2j} &\text{  for  } i< j \leq m.
            \end{cases}\]   
\end{proof}

 \bigskip

\section{$2$-Laplacians of simple graphs} 
  \label{sec-$2$-Laplacian of simple graphs}

In this section, we focus on the $2$-Laplacian. 
We compute explicitly the eigenvalues and eigenvectors of the $2$-Laplacians of some families of graphs, thereby comparing the properties of the $2$-Laplacian and the ordinary Laplacian of a simple graph.
 
\subsection{Definition and basic properties}

Let us begin by restating the definition of the $2$-Laplacian of a simple graph. 
 
\begin{remark} \label{2-lap}
Let $G$ be a simple graph with vertices $v_1, ..., v_n$. The \underline{$2$-Laplacian} of $G$ in Definition \ref{m-lap} with $m=2$ is the same as
\[
D' - \frac{1}{12}\left(16A_{G}-A_{G}^2 + D_{G}\right)
\]
where $D'$ is the diagonal matrix whose $(i,i)$ 
entry is the sum of all entries in the $i$th row of the matrix
\[
A':=\frac{1}{12}\left(16A_{G}-A_{G}^2 + D_{G}\right).
\]
\end{remark}

Note that the adjacency matrix of $G_2$ is $A'$. 
For the remainder of this paper, we will simply refer to $G_2$ as $G'$.

\medskip

We first investigate the $2$-Laplacians of regular graphs.

\begin{proposition} 
Let $G$ be a $k$-regular graph with $n$ vertices. 
If $\gamma_1 \geq \gamma_2 \geq \cdots \geq \gamma_n$ 
are the eigenvalues of its adjacency matrix $A_G$, 
then the eigenvalues of $L^{(2)}_G$ are
\[ 
\lambda_i = \frac{1}{12}\left( \gamma_i^2 -16 \gamma_i +16k-k^2 \right) 
\]
for $1 \leq i \leq n$.
\end{proposition}
\begin{proof}
It follows directly from
    \begin{align*}
        L^{(2)}_G &= L_{G'} \\
        &= \frac{1}{12}\left\{ 16kI_n - 16A_G - (k^2I_n - A_G^2) \right\} \\
        &= \frac{1}{12}\left\{ (16k-k^2)I_n - 16A_G + A_G^2 \right\}. 
    \end{align*}
\end{proof}

We note that the eigenvalues of $L_G^{(2)}$ can be negative. Thus, unlike the ordinary Laplacian, the $2$-Laplacian is not postive semi-definite in general. See Corollary \ref{complete graph n=18}. 

\begin{corollary} \label{positive semidefinity}
If $G$ is a $k$-regular graph with $k \leq 8$ 
then the $2$-Laplacian $L^{(2)}_G$ is positive semi-definite.
\end{corollary}

\bigskip

Next, we want to compute the $2$-Laplacians under standard graph operations. Let us recall definitions first.
Given two graphs $G$ and $H$ with vertex sets $V(G)$ and $V(H)$, the \underline{Cartesian product} $G \Box H$ of $G$ and $H$ is the graph with the vertex set $V(G) \times V(H)$ such that two vertices $(v,w)$ and $(v',w')$ are adjacent if $v = v'$ and $w$ is adjacent to $w'$ in $H$, or if $w = w'$ and $v$ is adjacent to $v'$ in $G$. 
The \underline{tensor product} $G \times H$ of $G$ and $H$ is the graph with the vertex set $V(G) \times V(H)$ such that two vertices $(v,w)$ and $(v',w')$ are adjacent if $v$ is adjacent to $v'$ in $G$ and $w$ is adjacent to $w'$ in $H$.

\begin{proposition}
 Let $G \Box H$ be the Cartesian product of two graphs $G$ and $H$ with $n_{1}$ and $n_{2}$ vertices, respectively. Then,
 \[L^{(2)}_{G \Box H}=L^{(2)}_G \otimes I_{n_2}+I_{n_1}\otimes L^{(2)}_H-\frac{1}{6}L_{G\times H}, \]
where $G\times H$ is the tensor product of two graphs $G$ and $H$.
\end{proposition}
\begin{proof}
    Let $K=G \Box H$, then
    \begin{align*}
        A_{K'}&=\frac{4}{3}A_K-\frac{1}{12}\left(A_K^2-D_K\right)\\
        &=\frac{4}{3}\left(A_G\otimes I_{n_2}+I_{n_1}\otimes A_H\right)\\
        &-\frac{1}{12}\left(A_G^2\otimes I_{n_2}+I_{n_1}\otimes A_H^2+2A_G\otimes A_H-D_K\right).
    \end{align*}
    Since $A_{G \times H}=A_G\otimes A_H$,
    \begin{align*}
        L^{(2)}_K&=L_{K'} \\
        &=L^{(2)}_G \otimes I_{n_2}+I_{n_1}\otimes L^{(2)}_H-\frac{1}{6}L_{G\times H}.
    \end{align*}
\end{proof}

Recall that 
\[ L_{G \Box H}=L_G \otimes I_{n_2}+I_{n_1}\otimes L_H. \]
See \cite{Fiedler}.
We can see that the $2$-Laplacian version is similar to this except for just one term.

\medskip

The following proposition shows that the $2$-Laplacian of the complement graph can be expressed in terms of the $2$-Laplacian and ordinary Laplacian of the given graph.
It is known that
\[ L_{G}=-L_{\Bar{G}}+nI_n -J_n, \]
where $J_{n}$ is a $n \times n$ matrix with all entries $1$.

\begin{proposition}
Let $G$ be a simple graph with vertices $v_1, ..., v_n$, and $d_{i}$ be the degree of $v_i$. If $\Bar{G}$ is the  complement graph of $G$ then
\[
L^{(2)}_{\Bar{G}}=L^{(2)}_G-\frac{17}{6}L_G+\frac{18-n}{12}L_{K_{n}}+\frac{1}{12} M
\]
where $M$ is a $n \times n$ matrix whose $(i,j)$ entry is $d_{i}+d_{j}$.
\end{proposition}
\begin{proof}
    Note that
    \begin{align*}
        A_{\Bar{G}}^2&=(A_{K_{n}}-A_G)^2\\
        &=A_{K_{n}}^2-A_{K_{n}}A_G-A_G A_{K_{n}}+A_G^2\\
        &=nI+(n-2)A_{K_{n}}-(A_{K_{n}}A_G+A_G A_{K_{n}})+A_G^2.
    \end{align*}
    Since $A_{K_{n}}A_G+A_G A_{K_{n}}$ and $M -2A_G$ differ only in diagonal entries, we obtain 
    \[
    L^{(2)}_{\Bar{G}}=L^{(2)}_G-\frac{17}{6}L_G+\frac{18-n}{12}L_{K_{n}}+\frac{1}{12} M 
    \]
 where $M_{i,j}= d_{i}+d_{j}$.
\end{proof}

\subsection{Laplacian and $2$-Laplacian cospectral mates} 

For all graphs $G$ with no more than $9$ vertices, we computed both the number and the proportion of graphs with cospectral mates for $L^{(2)}_G$ via SageMath \cite{sagemath} and listed in Table \ref{tab-cospectral number} and Table \ref{ratio of cospectral}, respectively. 
We also listed the number of such graphs for the adjacency matrix $A_G$ and the Laplacian matrix $L_G$ from \cite{BC} and signless Laplacian $|L_G|$ from \cite{H cospectral}.
It is noteworthy that the number of cospectral graphs with respect to $L^{(2)}_G$ is significantly smaller than 
all others.

\begin{table} 
\begin{center}
\begin{tabular}{|c|c|c|c|c|c|}
\hline
\hspace*{1.5mm}$|V|$ \hspace*{1.5mm}& \hspace*{1.5mm} $G$ \hspace*{1.5mm} & \hspace*{1.5mm} $A$ \hspace*{1.5mm}& \hspace*{1.5mm} $L$ \hspace*{1.5mm} & \hspace*{1.5mm} $|L|$\hspace*{1.5mm} & \hspace*{1.45mm} $L^{(2)}$ \hspace*{1.45mm} \\
\hline
1  &  1 &    0  &   0   &  0 &   0         \\
2  &  2 &    0  &   0   &  0 &   0         \\
3  &  4 &    0  &   0   &  0  &  0         \\
4  &  11 &   0   &  0    & 0.18181  &  0          \\
5  &  34 &   0.05882   &  0    &  0.11765  &   0        \\
6  &  156 &  0.06410    &   0.02564   &  0.10256  &   0        \\
7  &  1044 &  0.10536    &   0.12452   & 0.09770 &   0.00192          \\
8  &  12346 &  0.13948    &   0.14312   & 0.09728  & 0.00016           \\
9  &  274668 &  0.18582    &   0.15508   &  0.06918 & 0.00001          \\
\hline
\end{tabular}
\end{center}
\caption{Proportions of graphs with cospectral mates}
\label{ratio of cospectral}
\end{table}

As the value of $n$ increases up to $9$, we observe a decrease in the fraction of graphs with cospectral mates. 
This observation suggests that it may be more effective to analyze the structure of the graph using $2$-Laplacian spectra instead of the others. 

\medskip

We first give an example of a pair of graphs that are cospectral with respect to $L_{G}$ but 
not with respect to $L^{(2)}_{G}$.

\begin{example} \label{bip example}
The following two graphs $G$ and $H$ are cospectral 
with respect to the ordinary Laplacians. See \cite{VW}. 
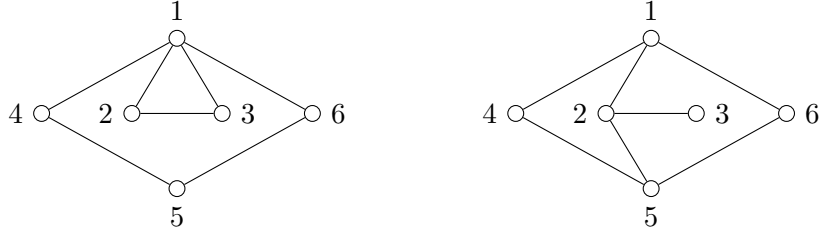
\[
\begin{tikzpicture}[x=.6cm, y=1cm]
	\vertex (1) at (3,1) [label=above:$1$] {};
	\vertex (2) at (2,0) [label=left:$2$] {};
	\vertex (3) at (4,0) [label=right:$3$] {};
	\vertex (4) at (0,0) [label=left:$4$] {};
    \vertex (5) at (3,-1) [label=below:$5$] {};
    \vertex (6) at (6,0) [label=right:$6$] {};
	\path
		(1) edge (2)
		(1) edge (3)
		(1) edge (4)
  (1) edge (6)
  (2) edge (3)
  (4) edge (5)
  (5) edge (6)
	;
\end{tikzpicture} \qquad \qquad 
\begin{tikzpicture}[x=.6cm, y=1cm]
	\vertex (1) at (3,1) [label=above:$1$] {};
	\vertex (2) at (2,0) [label=left:$2$] {};
	\vertex (3) at (4,0) [label=right:$3$] {};
	\vertex (4) at (0,0) [label=left:$4$] {};
    \vertex (5) at (3,-1) [label=below:$5$] {};
    \vertex (6) at (6,0) [label=right:$6$] {};
	\path
		(1) edge (2)
		(1) edge (4)
  (1) edge (6)
  (2) edge (3)
  (2) edge (5)
  (4) edge (5)
  (5) edge (6)
	;
\end{tikzpicture}\]
On the other hand, their $2$-Laplacians are 
\[
 L^{(2)}_G=
\frac{1}{12}\begin{bmatrix}
60 & -15 & -15 & -16  & 2 & -16 \\
-15 & 28 & -15 & 1  & 0 & 1 \\
-15 & -15  & 28 & 1 & 0 &  1 \\
-16 & 1 & 1 & 28 & -16 & 2 \\
2 & 0 & 0 & -16  & 30 & -16 \\
-16  & 1 & 1 & 2 & -16 & 28
\end{bmatrix},\]
\[
 L^{(2)}_H=
\frac{1}{12}\begin{bmatrix}
44 & -16 & 1 & -16  & 3 & -16 \\
-16 & 44 & -16 & 2  & -16 & 2 \\
1 & -16  & 14 & 0 & 1 &  0 \\
-16 & 2 & 0 & 28 & -16 & 2 \\
3 & -16 & 1 & -16  & 44 & -16 \\
-16  & 2 & 0 & 2 & -16 & 28
\end{bmatrix}
\]
and they have different spectra: 
\begin{align*}
    spec(L^{(2)}_G)&= \left\{0, \frac{7}{10}, \frac{13}{6}, \frac{43}{12}, \frac{437}{120}, \frac{27}{4}\right\},\\
    spec(L^{(2)}_H)&=\left\{0, \frac{3}{4}, \frac{13}{6}, \frac{41}{12}, \frac{71}{20}, \frac{833}{120}\right\}.
\end{align*}
\end{example}

Note that the first graph in Example \ref{bip example} is not bipartite, while the second one is bipartite.
It is well known that the spectrum of the ordinary Laplacian does not determine whether a graph is bipartite or not. See, for example, \cite{B const of cospectral} and \cite{C spectra}. 
However, based on the above example, the $2$-Laplacian spectrum may provide some insight into this determination.

\medskip

As indicated in Table \ref{tab-cospectral number}, however, not all graphs are distinguished by their $L^{(2)}$-spectra. Here is an example of a $2$-Laplacian cospectral pair.

\begin{example} The following two graphs are cospectral with respect to the $2$-Laplacians. One can show that 
they are also cospectral with respect to the ordinary Laplacians as well.
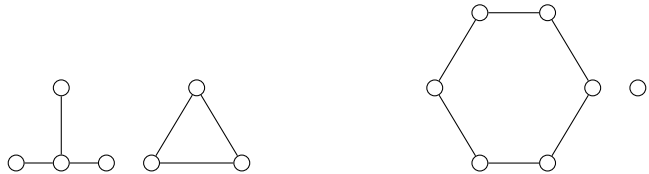
\[
\begin{tikzpicture}[x=.6cm, y=1cm]
	\vertex (1) at (0,0)  {};
	\vertex (2) at (0,1)  {};
	\vertex (3) at (-1,0) {};
	\vertex (4) at (1,0)  {};
    \vertex (5) at (3,1)  {};
    \vertex (6) at (2,0)  {};
    \vertex (7) at (4,0)  {};

	\path
		(1) edge (2)
		(1) edge (3)
		(1) edge (4)
  (5) edge (6)
  (5) edge (7)
  (6) edge (7)
	;
\end{tikzpicture} \qquad \qquad \qquad 
\begin{tikzpicture}[x=.6cm, y=1cm]
	\vertex (1) at (4,0)  {};
	\vertex (2) at (3,1)  {};
	\vertex (3) at (1.5,1) {};
	\vertex (4) at (0.5,0)  {};
    \vertex (5) at (1.5,-1)  {};
    \vertex (6) at (3,-1)  {};
    \vertex (7) at (5,0)  {};

	\path
		(1) edge (2)
		(2) edge (3)
		(3) edge (4)
        (4) edge (5)
		(5) edge (6)
		(6) edge (1)
	;
\end{tikzpicture}
\]
\end{example}

\medskip

\medskip

\subsection{$2$-Laplician spectra of circulant graphs} 

In the rest of this section, we compute the $2$-Laplacian spectra of circulant graphs.
This allows us to illustrate the similarities and differences between the $2$-Laplacian and ordinary Laplacian spectra.

\begin{notation} 
To define graphs of circulant types, let us fix some notation first.

\begin{enumerate}
\item All our matrices and vectors are defined over real numbers.
\item $n$-vectors are $n \times 1$ matrices.
\item For a $n \times n$ matrix $M$, we let $(M)_{ij}$ denote the $(i,j)$ entry 
of $M$ for $0 \leq i , j \leq n-1$. 
\item For a $n$-vector $x$, we let $(x)_k$ or $x_k$ denote the $k$th entry 
of $x$ for $0 \leq k < n$.
The indices $k$ for the entries of $x$ will be regarded as elements of $\mathbb{Z}_n$, 
thus 
\[
x_{n+k} = x_k \quad \text{for all $k \in \mathbb{Z}$}.
\]
\end{enumerate}
\end{notation}

For every column vector 
\[
x=\begin{bmatrix} x_0 \\ x_1 \\ \vdots \\ x_{n-1} \end{bmatrix} 
\]
we can define the {circulant matrix} generated by $x$ as 
\[
C_{x} = \begin{bmatrix}
x_{0} & x_{n-1} & \cdots & x_{1} \\
x_{1} & x_{0} & \cdots & x_{2} \\
\vdots & \vdots & \ddots & \vdots \\
x_{n-1} & x_{n-2} & \cdots & x_{0}
\end{bmatrix}.
\]

The following results are well known. 
See, for example, \cite{davis1979circulant} and \cite{gray2006toeplitz}.
\begin{lemma}
Let $X$ and $Y$ be $n \times n$ circulant matrices. Then 
$X+Y$ and $XY$ are circulant, and $XY=YX$. If  
$X$ is invertible then its inverse $X^{-1}$ is also circulant.
\end{lemma}

\medskip
 
\begin{definition}
For a positive integer $n$, let $s_1, ..., s_k$ be integers such that 
\[
1 \leq s_{1} < s_{2} < \cdots < s_{k} \leq \frac{n}{2}.
\]
The \underline{circulant graph} $\operatorname{Circ}_n(s_1,s_2,...,s_k)$ is a graph with $n$ vertices 
labeled as $0,1,...,n-1$ such  that 
each vertex $i$ is adjacent to vertices 
$i \pm s_{j}$ (mod $n$) for all $1\leq j \leq k$.
\end{definition}

Note that a circulant graph $\operatorname{Circ}_n(s_1,s_2,...,s_k)$ is either $2k$-regular or $(2k-1)$-regular. It is $(2k-1)$-regular if and only if $n$ is even and $s_k=\frac{n}{2}$.

\begin{example} 
Here are some well-known examples of circulant graphs. 
See also Figures \ref{fig-3-antiprism}, \ref{fig-3-prism}, and \ref{fig-3-Moebius}.
\begin{enumerate}
\item The complete graph $K_n=\operatorname{Circ}_{n}(1,2,...,[n/2])$.

\item The complete bipartite graph $K_{n,n}=\operatorname{Circ}_{2n}(1,3,5,...,2[n/2]-1)$.

\item $n$-antiprism graph $\operatorname{Circ}_{2n}(1,2)$.

\item $n$-prism graph $Y_n=\operatorname{Circ}_{2n}(2,n)$ for odd number $n$.
        
\item The M{\"o}bius ladder graph $M_n=\operatorname{Circ}_{2n}(1,n)$.
\end{enumerate}
\end{example}

\begin{figure}    
 \[
 \begin{tikzpicture}[x=1.2cm, y=1.2cm]
	\vertex (1) at (0:1) [label=right:$1$]{};
    \vertex (2) at (300:1) [label=below:$2$]{};
	\vertex (3) at (240:1) [label=below:$3$]{};
	\vertex (4) at (180:1) [label=left:$4$]{};
	\vertex (5) at (120:1) [label=above:$5$]{};
	\vertex (0) at (60:1) [label=above:$0$]{};
	\path 
		(1) edge (2) 
        (2) edge (3)
        (3) edge (4)
        (4) edge (5)
        (5) edge (0)
        (0) edge (1)
        (1) edge (3)
        (2) edge (4)
        (3) edge (5)
        (4) edge (0)
        (5) edge (1)
        (0) edge (2)
	;
\end{tikzpicture}
\]
\caption{$3$-antiprism graph $\operatorname{Circ}_6(1,2)$}
\label{fig-3-antiprism}
\end{figure}

\begin{figure}    
 \[
 \begin{tikzpicture}[x=1.2cm, y=1.2cm]
	\vertex (1) at (0:1) [label=right:$1$]{};
    \vertex (2) at (300:1) [label=below:$2$]{};
	\vertex (3) at (240:1) [label=below:$3$]{};
	\vertex (4) at (180:1) [label=left:$4$]{};
	\vertex (5) at (120:1) [label=above:$5$]{};
	\vertex (0) at (60:1) [label=above:$0$]{};
	\path 
		(1) edge (3) 
        (2) edge (4)
        (3) edge (5)
        (4) edge (0)
        (5) edge (1)
        (0) edge (2)
        (1) edge (4)
        (2) edge (5)
        (3) edge (0)
	;
\end{tikzpicture}
\]
\caption{$3$-prism graph $Y_3=\operatorname{Circ}_6(1,2)$}
\label{fig-3-prism}
\end{figure}

\begin{figure}    
 \[
 \begin{tikzpicture}[x=1.2cm, y=1.2cm]
	\vertex (1) at (0:1) [label=right:$1$]{};
    \vertex (2) at (300:1) [label=below:$2$]{};
	\vertex (3) at (240:1) [label=below:$3$]{};
	\vertex (4) at (180:1) [label=left:$4$]{};
	\vertex (5) at (120:1) [label=above:$5$]{};
	\vertex (0) at (60:1) [label=above:$0$]{};
	\path 
		(1) edge (2) 
        (2) edge (3)
        (3) edge (4)
        (4) edge (5)
        (5) edge (0)
        (0) edge (1)
        (1) edge (4)
        (2) edge (5)
        (3) edge (0)
	;
\end{tikzpicture}
\]
\caption{M{\"o}bius ladder graph $M_3=\operatorname{Circ}_6(1,3)$}
\label{fig-3-Moebius}
\end{figure}

\begin{theorem}
Let $G$ be a $k'$-regular circulant graph $\operatorname{Circ}_n(s_1,s_2,...,s_k)$.
Then the eigenvectors of $L^{(2)}_{G}$ are 
\[ 
\mathbf{v}_{j}=(1, \omega^j, \omega^{2j},...,\omega^{(n-1)j})^T \quad \text{for $0 \leq j \leq n-1$}
\]
where $\omega=\exp(\frac{2\pi i}{n})$, and 
the corresponding eigenvalues are 
\[ 
\frac{1}{12} \left\{  \lambda_{j}^2 -16\lambda_j -(k')^2+16k' \right\}  
\]
where $\lambda_j$'s are the eigenvalues of 
the matrix $C_x$ given in Lemma \ref{circulant spectrum}. Thus,
\[ 
\lambda_{j}= 
\begin{cases}
    (-1)^j+\sum_{i=1}^{k-1}  2\cos{\frac{2\pi s_{i} j}{n}}   &\text{  if $n$ is even and $s_k=\frac{n}{2}$}\\
    \sum_{i=1}^{k} 2\cos{\frac{2\pi s_{i} j}{n}} &\text{  otherwise.}
\end{cases}     
\]  
\end{theorem}
\begin{proof}
Note that the adjacency matrix of $G$ is a circulant matrix $C_{x}$,  where 
\[ 
x_{j}= \begin{cases}
1 & \text{if $j=\pm s_{\ell}$ (mod $n$) for some $\ell$} \\
0  &\text{otherwise.}
\end{cases}
\]
Since $G$ is $k'$-regular, we have
\begin{align*}
    A_{G'}&=\frac{1}{12}\left( 16A_{G}-A_{G}^2+D_{G} \right)\\
    &=\frac{1}{12}\left( 16C_{x}-C_{x}^2+k' I_{n} \right),
\end{align*}
and thus
\begin{align*}
    L^{(2)}_G &= L_{G'} \\
    &= \frac{1}{12}\left\{ 16(k'I_{n}-C_x)-(k')^2 I_n +C_x^2 \right\}\\
    &= \frac{1}{12}\left\{ (-(k')^2+16k')I_n -16C_x +C_x^2 \right\}.
\end{align*}
Note that $C_x^2$ is also a circulant matrix. 
Therefore, by Lemma \ref{circulant spectrum}, the eigenvectors of $L^{(2)}_{G}$ are 
\[ 
\mathbf{v}_{j}=(1, \omega^j, \omega^{2j},...,\omega^{(n-1)j})^T \quad \text{for $0 \leq j \leq n-1$}
\]
where $\omega=\exp(\frac{2\pi i}{n})$ and 
the corresponding eigenvalues are 
\[ 
\frac{1}{12} \left\{  \lambda_{j}^2 -16\lambda_j -(k')^2+16k' \right\}  
\]
where $\lambda_j$'s are the eigenvalues of $C_x$.
\end{proof}

\medskip

\begin{corollary} \label{2-lap spec of complete graph}
The eigenvectors of $L^{(2)}_{K_{n}}$ are 
\[ 
\mathbf{v}_{j}=(1, \omega^j, \omega^{2j},...,\omega^{(n-1)j})^T \quad \text{for $0 \leq j \leq n-1$}
\]
where $\omega=\exp(\frac{2\pi i}{n})$.
The corresponding eigenvalues are 
\[ 
\lambda_{j}= 
\begin{cases}
0 &\text{  if  } j=0 \\
\frac{1}{12}n(18-n)  &\text{  otherwise.}
\end{cases}
\]     
\end{corollary}

From Corollary \ref{2-lap spec of complete graph}, 
we see that the multiplicity of zero eigenvalues of the $2$-Laplacian matrix does not count  
the number of connected components---the 
complete graph $K_{18}$ is connected, but 
the multiplicity of zero eigenvalues is $18$.
Also, Corollary \ref{2-lap spec of complete graph} gives 
a more accurate version of Corollary \ref{positive semidefinity} for complete graphs.

\begin{corollary} \label{complete graph n=18}
The $2$-Laplacian $L^{(2)}_{K_{n}}$ of $K_n$ is positive 
semi-definite if and only if $1 \leq n \leq 18$.
\end{corollary}

\medskip

Next, seeing the second smallest eigenvalue of the ordinary 
Laplacian plays an important roles in spectral graph theory, we compare it with the second smallest eigenvalue of 
the positive semi-definite $2$-Laplacian. 
Let us show a technical lemma first.

\begin{lemma}\label{courant}
Suppose that $x_{1}+x_{2}+\cdots+x_{n}=0$ and $x_{1}^2+x_{2}^2+\cdots+x_{n}^2=1$. Then,  
\[
\sum_{1 \leq i<j\leq n} (x_{i}-x_{j})^2=n.
\]
\end{lemma}

\begin{proof}
    \begin{align*}
        \sum_{1 \leq i<j\leq n} (x_{i}-x_{j})^2 &=\sum_{1 \leq i<j\leq n-1} (x_{i}-x_{j})^2 + \sum_{1 \leq i\leq n-1} (x_{i}-x_{n})^2\\
        &=\sum_{1 \leq i<j\leq n-1} (x_{i}-x_{j})^2  + (n+1)x_{n}^2+x_{1}^2+x_{2}^2+\cdots + x_{n-1}^2 \\
        &=2n\cdot (x_{1}^2+x_{2}^2+\cdots + x_{n-1}^2 + \sum_{1 \leq i<j\leq n-1}{x_{i}x_{j}})\\
        &=n\cdot \{x_{1}^2+x_{2}^2+\cdots + x_{n-1}^2 +(x_{1}+x_{2}+\cdots + x_{n-1})^2 \}\\
        &=n\cdot (x_{1}^2+x_{2}^2+\cdots + x_{n}^2 )\\
        &=n.
\end{align*}
\end{proof}

\begin{theorem}\label{thm:second-smallest-L2}
Let $G$ be a simple graph with $n$ vertices whose $2$-Laplacian $L^{(2)}_G$ is positive semi-definite. 
If $\lambda_{2}$ and $\mu_2$ are the second smallest eigenvalues of $L^{(2)}_G$ and $L_G$ respectively, then 
\[
\frac{4}{3}\mu_{2}-\frac{n(n-2)}{6} \; 
\leq 
\;
\lambda_{2}\;\leq \;\frac{4}{3}\mu_{2}.
\]
\end{theorem}
\begin{proof} 
Note that by the Courant-Fischer theorem,
\[
\lambda_{2}=\min\{x^{T}L^{(2)}_G x : x \perp (1,...,1)^T \text{ and } x^{T}x=1\}.
\]

For any $n$-vector $x$ satisfying $x \perp (1,...,1)^T$ and $x^{T}x=1$, since
\begin{align*}
x^{T}L^{(2)}_G x &=a_{1,2}\sum_{1 \leq i<j\leq n} (A_G)_{ij}(x_{i}-x_{j})^2+a_{2,2}\sum_{1 \leq i<j\leq n} (A_G^2)_{ij}(x_{i}-x_{j})^2\\
        &=\frac{4}{3}\sum_{1 \leq i<j\leq n} (A_G)_{ij}(x_{i}-x_{j})^2-\frac{1}{12}\sum_{1 \leq i<j\leq n} (A_G^2)_{ij}(x_{i}-x_{j})^2
\end{align*}
and $0 \leq (A_G^2)_{ij} \leq n-2$ for each $i$ and $j$, 
\[
\frac{4}{3}\sum_{i\sim j}(x_{i}-x_{j})^2-\frac{n-2}{12}\sum_{1 \leq i<j\leq n} (x_{i}-x_{j})^2 
\; \leq \; 
x^{T}L^{(2)}_G x 
\; \leq \;
\frac{4}{3}\sum_{i\sim j}(x_{i}-x_{j})^2.
\]
Here $i\sim j$ means $i$ and $j$ are adjacent with $i < j$.
 Using
     \begin{align*}
         \mu_{2}&=\min\{x^{T}L_G x : x \perp (1,...,1) \text{ and } x^{T}x=1\}\\
         &=\min\{\sum_{i\sim j}(x_{i}-x_{j})^2 : x \perp (1,...,1) \text{ and } x^{T}x=1\}
     \end{align*}
    and Lemma \ref{courant}, we have
    \[\frac{4}{3}\mu_{2}-\frac{n(n-2)}{6} \;\leq \; \lambda_{2} \; \leq \; \frac{4}{3}\mu_{2}.\]
\end{proof}


\bigskip

\end{document}